\documentclass{article}
\usepackage{amsfonts,amsmath,amssymb,amsthm,color}

\def\R{{\mathbb R}}
\def\N{{\mathbb N}}

\def\diver{{\rm div}}
\def\divg{\diver_\gamma}

\def\<{\langle}
\def\>{\rangle}
\def\eps{\varepsilon}

\def\phi{\varphi}
\def\1{{\bf 1}}

\newcommand\scal[2]{{\left\langle #1 ,#2\right\rangle}}
\newcommand{\essinf}{\operatorname{ess\, inf}}
\newcommand{\esssup}{\operatorname{ess\, sup}}
\newtheorem{thm}{Theorem}

\newtheorem{prop}[thm]{Proposition}
\newtheorem{cor}[thm]{Corollary}
\newtheorem{lemma}[thm]{Lemma}
\newtheorem{rmk1}[thm]{Remark}

\title{Perimeter of sublevel sets in infinite dimensional spaces}

\author{Vicent Caselles\thanks{
Departament de Tecnologies de la Informaci\'o i les Comunicacions, Universitat Pompeu--Fabra, C/Roc Boronat 138,
08018 Barcelona, Spain,
e--mail: vicent.caselles@upf.edu
},
Alessandra Lunardi \thanks{
Dipartimento di Matematica, University of Parma, Parco Area delle Scienze 53/a, 43124 Parma, Italy,
e--mail: alessandra.lunardi@unipr.it
},
Michele Miranda jr \thanks{
Dipartimento di Matematica, University of Ferrara, via Machiavelli 35, 44121 Ferrara, Italy,
e--mail: michele.miranda@unife.it
}, Matteo Novaga\thanks{
Dipartimento di Matematica, University of Padova, via Trieste 63, 35121 Padova, Italy,
e--mail: novaga@math.unipd.it
}}

\date{}

\begin{document}
\maketitle

\begin{abstract}
\noindent We compare the perimeter measure with the Airault-Malliavin surface measure
and we prove that all open convex subsets of abstract Wiener spaces 
have finite perimeter. By an explicit counter--example, we show that in general this 
is not true for compact convex domains.
\end{abstract}

\section{Introduction}\label{secIntro}

In the setting of abstract Wiener spaces and  Malliavin calculus, the definition of 
set with finite perimeter and function of bounded variation has been first given
by Fukushima and Hino in \cite{fuk2000}, \cite{Fuk01OnT}.
Recently there has been an increasing interest in the study of geometric properties
of   sets with finite perimeter and, in particular, in the structure of the perimeter
measure. We mention for instance the paper by Hino \cite{Hin09}, where
the author provides a notion of cylindrical essential boundary and a representation
of the perimeter measure by means of a codimension one Hausdorff measure, introduced
by Feyel and de la Pradelle in \cite{feydelpra}.
In the papers \cite{AmbManMirPal09BVF,AMP} Ambrosio et al. give a new version of these results,
together with the Sobolev rectifiability of the essential boundary, that is the fact that
the essential boundary is contained, up to negligible sets, in a countable union of graphs
of Sobolev functions defined on hyperplanes. The question whether or not the rectifiability 
result can be extended, as in the Euclidean case, to Lipschitz functions is still an open question.

In this paper we address some questions about perimeters of good sets. First, we compare
the perimeter measure with the surface measure introduced by Airault and Malliavin in \cite{AM},
 showing that for suitably smooth sets such notions coincide.
We establish the equality  
\begin{equation}
\label{per-div}
P_{\gamma}(\{ u <r\} ) = \int_{\{u<r\}}   {\rm div}_\gamma \bigg(\frac{ \nabla_{H}u }{|\nabla_{H}u|_H} \bigg)\, d\gamma , \quad r\in \R,
\end{equation}
for a wide class of real valued functions $u$. Here $P_{\gamma}$ and $  {\rm div}_\gamma$ denote the perimeter and the divergence with respect to the Gaussian measure $\gamma$, and $H$ is the relevant Cameron--Martin space, see Sect. 2 for precise definitions. Formula \eqref{per-div} has several consequences, such as continuity and boundedness of $r\mapsto P_{\gamma}(\{ u <r\} )$. 

Then, we investigate the question whether or not a convex set  has finite perimeter. In Proposition 
\ref{propen} we show that all open convex sets have finite perimeter.
On the other hand, in Proposition \ref{PropConvC} we prove that in any infinite dimensional 
Hilbert space  with a non--degenerate Gaussian measure 
there exists a closed convex set (a Hilbert cube) with infinite perimeter.
Such a convex set is compact under a mild condition on the covariance operator.

In the case of balls related results may be found in the papers \cite{HSD,hertle,linde},
where the notion of perimeter is replaced by the density (with respect to the Lebesgue measure) of 
the image measure of $\gamma$ under $\|\cdot - x_0\|$, and that contain  further discussions of other 
aspects of Gaussian measures of balls. In a Hilbert space, taking $u(x) = \|x-x_0\|^2$, \eqref{per-div} gives 
a simple explicit formula for the the perimeter of any ball.

\section{Notation and preliminary results}\label{secnot}

We consider an abstract Wiener space $(X,\gamma,H)$,   where
$X$ is a separable Banach space, endowed with the norm $\|\cdot\|_X$,
$\gamma$ is a non--degenerate centered Gaussian measure,
and $H$ is the Cameron--Martin space associated to the measure $\gamma$. 

Let us recall the definition and properties of $H$ that will be used in the sequel. 
By Fernique's Theorem  (e.g., \cite[Theorem 2.8.5]{Bog98Gau}),
 there exists a positive number $\beta>0$ such that
\[
 \int_X e^{\beta\|x\|^2}d\gamma(x)<+\infty.
\]
This implies that the dual space $X^*$ is contained in $L^2(X, \gamma)$. The closure ${\cal H}$ of $X^*$  in 
$L^2(X, \gamma)$ is called {\it reproducing kernel}, and $H$ is the range of the one to one operator $R: {\cal H}\to X$
defined by
\[
Rf := \int_X f(x)x \,d\gamma(x),
\]
(the latter is a Bochner integral). 
$H$ is endowed with the inner product
$[\cdot, \cdot ]_H$ and  the associated norm $|\cdot |_H$ induced by $L^2(X, \gamma)$ through $R$. So, 
$h\in H$ if and only if there is $\hat{h}\in  {\cal H}$ such that 
$$
\int_X \hat{h}(x) \scal{x}{x^*} d\gamma(x) = \scal{h}{x^*}, \qquad \forall x^*\in X^*,
$$ 
(here $\scal{\cdot}{\cdot}$ denotes the duality between $X$ and $X^*$),
and in this case $|h |_H = \|\hat{h}\|_{L^2(X, \gamma)}$.

The continuity of $R$ implies that the embedding of $H$ in $X$ is continuous, that is, there exists $c_H>0$
such that
\begin{equation}\label{contEmb}
\|h\|_X \leq c_H|h|_H,\qquad \forall h\in H.
\end{equation}
Moreover, $H$ is separable and it is densely embedded in $X$;  there exists a 
sequence $(x^*_j)$ in $X^*$, such  that the elements $h_j := Rx_j^*$, $j\in \N$,  form an orthonormal basis
of $H$. We then define $\lambda_j:=\|x^*_j\|_{X^*}^{-2}$.  We shall consider an ordering of  the 
vectors $x^*_j$ such  that the   sequence $(\lambda_j)$ is non increasing.

Given $n\in\mathbb N$, we denote by  $H_n$ the linear span of $h_1,\ldots, h_n$, 
and by $\Pi_n: X\to H_n$   the  projection
\[
\Pi_n(x) := \sum_{j=1}^n \scal{x}{x^*_j}\, h_j,  \qquad x\in X.
\]
The map $\Pi_n$ induces the decomposition $\gamma=\gamma_n\otimes\gamma_n^\perp$,
with $\gamma_n$ and $\gamma_n^\perp$ Gaussian measures having $H_n$ and $H_n^\perp$
as Cameron--Martin spaces.

The covariance operator   $Q$ is the restriction of $R$ to $X^*$. 
In the case that $X$ is a Hilbert space, after the canonical identification of $X$ and $X^*$, $Q$ is a 
bounded symmetric trace class operator in $X$ and 
the constant $c_H$ is related to the largest  eigenvalue of  $Q$, $c_H = \lambda_{1}^{1/2}$. The numbers 
$\lambda_j$ considered  above are precisely the eigenvalues of $Q$. 

The measure $\gamma$ is absolutely continuous with respect
to translations along Cameron--Martin directions; more precisely, for $h\in H$, $h=Rx^*$,
the measure
$\gamma_h(B)=\gamma(B-h)$ is absolutely continuous with respect to $\gamma$ and
\begin{equation}\label{absCont}
d\gamma_h(x)=\exp\left(
\scal{x}{x^*}-\frac{1}{2}|h|_H^2
\right)d\gamma(x) .
\end{equation}

For any function $f:X\mapsto \R$ differentiable at a point $x\in X$,  the derivative  $f'(x)$  is an 
element of $X^*$, hence its restriction to $H$  belongs to $H^*$. The element $y\in H$ such that  
$f'(x)(h) = [y, h]_H$ for each $h\in H$ is denoted by $ \nabla_H f(x)$. It follows that 
$$\nabla_H f(x) =  \sum_{j\in\mathbb N}\partial_jf(x)h_j,$$
where $\partial_j := \partial_{h_j}$ is the directional derivative of $f$ in the direction $h_j$. 

We denote by ${\cal F}C_b^1(X)$ the space 
\begin{align*}
{\cal F}C^1_b(X)=\{ f:X\to \R:&\,\exists m\in \N, \ell_1,\ldots,\ell_m\in X^*, \mbox{ such that}\\ 
&f(x)=\varphi (\scal{x}{\ell_1},\ldots,\scal{x}{\ell_m}),\varphi \in C^1_b(\R^m)\}.
\end{align*}
We also define the space ${\cal F}C^1_b(X,H)$ of cylindrical $H$--valued functions
as the vector space spanned by the functions 
$f\ell$, with $f\in {\cal F}C^1_b(X)$ and $\ell\in H$. 
For functions  $\varphi\in {\cal F}C^1_b(X,H)$,
$$
\varphi(x)=\sum_{i=1}^n f_i(x)\ell_i,  $$ 
with $f_i\in {\cal F}C^1_b(X)$, $\ell_i\in H$,
the divergence is defined as 
\begin{equation}\label{defOfDivergence}
\divg \varphi(x) =   \sum_{j\geq 1}  \partial_j^*[\varphi(x), h_j]_H
=\sum_{j\geq 1} \sum_{i=1}^n \partial_j^* f_i(x) [\ell_i,h_j]_H,
\end{equation}
where $\partial_j^*f(x) := \partial_jf(x) - \hat h_j(x)f(x)$; this divergence operator is, up
to the sign, the formal adjoint in $L^2(X,\gamma)$ of the gradient $\nabla_H$. In fact formula \eqref{defOfDivergence} may be extended to all vector fields $\varphi \in W^{1,p}(X, \gamma; H)$
and $\divg $ is a bounded operator from $W^{1,p}(X, \gamma; H)$ to $L^p(X, \gamma)$ for every $p\in (1, +\infty)$ 
\cite[Prop. 5.8.8]{Bog98Gau}. 

With these notations, the following integration by parts formula holds:
\begin{equation}\label{inp}
\int_X f\, \divg \phi\,d\gamma = -\int_X [\nabla_H f,\phi]_H\, d\gamma,
\quad \forall f\in {\cal F}C^1_b(X), \phi\in {\cal F}C^1_b(X,H).
\end{equation}
Moreover, if a function $u$ belongs to the Orlicz space  $L\log^{1/2}L(X, \gamma)$, then 
$u\, \divg  \phi\in L^1(X, \gamma)$ for each $ \phi\in {\cal F}C^1_b(X,H)$.

Following \cite{Fuk01OnT} and \cite{AmbManMirPal09BVF}, we
define the $\gamma$--total variation of a function $u\in L\log^{1/2}L(X, \gamma)$ as
\begin{equation}\label{defDu}
|D_\gamma u|(X) :=
\sup \left\{\int_X u(x)\divg \phi(x) d\gamma(x): \phi\in
{\cal F}C^1_b(X,H):
|\phi(x)|_H\leq 1
\right\}.
\end{equation}
We say that $u$ has finite $\gamma$--total variation, $u\in
BV(X,\gamma)$, if $|D_\gamma u|(X)<+\infty$. 
A measurable subset $E\subseteq X$ is said to have $\gamma$--finite
perimeter if $P_\gamma(E):=|D_\gamma \chi_E|(X)<+\infty$. The perimeter is lower semicontinuous
with respect to the $L^1$--convergence, in the sense that if $(E_n)$ is a sequence of sets with 
finite perimeter such that $\chi_{E_n}$ converges to $\chi_E$ in $L^1(X,\gamma)$, then  
$P_\gamma(E)\leq \lim\inf_{n\to \infty} P_\gamma(E_n)$.

To any function $u\in BV(X,\gamma)$ an $H$--valued  measure $D_\gamma u$ is associated. Thanks to the Radon--Nikodym theorem,  
the polar decomposition $D_\gamma u=\sigma_u|D_\gamma u|$ holds, where
$|D_\gamma u|$ is the total variation measure and $\sigma_u:X\to H$ is a $|D_\gamma u|$--measurable
function with $|\sigma_u(x)|_H=1$ $|D_\gamma u|$--a.e. $x\in X$. In the case $u=\chi_E$, we shall write
$D_\gamma \chi_E=\sigma_E |D_\gamma \chi_E|$.
For functions with bounded variation we have the following integration by parts formula, 
\begin{equation}\label{eqpappa}
\int_X u\,\divg \phi \,d\gamma = -\int_X [\phi,\sigma_u]_H d|D_\gamma u|,
\qquad \forall\phi\in {\cal F}C^1_b(X,H).
\end{equation}
In particular,  if $u\in W^{1,1}(X,\gamma)$,
then $u\in BV(X,\gamma)$, the total variation measure is absolutely
continuous with respect to the Gaussian measure, and 
$D_\gamma u=\nabla_H u \,\gamma$.

An important tool  
is the following coarea formula (\cite{fuk2000}, see also \cite{Led98}).
 
\begin{prop}\label{coarea}
Let $u\in BV(X,\gamma)$; then almost all the level sets
$\{u<r\}$ have finite perimeter and
the following  equality holds, 
\begin{equation}\label{eqcoarea}
|D_\gamma u|(X) = \int_\R P_\gamma\left(\{ u<r\}\right) \,dr.
\end{equation}
\end{prop}

\section{Comparison with the Malliavin surface measure}

We now compare the notion of perimeter with the surface measure 
introduced by Airault and Malliavin in \cite{AM} in the case of suitably smooth hypersurfaces. 
We refer to    \cite{AM} and to \cite[Sect. 6.9, 6.10]{Bog98Gau} for its construction and properties. 

The smooth surfaces under consideration are level surfaces of functions 
$$
u\in W^{\infty}(X, \gamma) :=
\bigcap_{p>1,\, k\in \N} W^{k,p}(X,\gamma)
$$ 
such that $\frac{1}{|\nabla_{H}
u|_H} \in \bigcap_{p > 1} L^p(X,\gamma)$. For such functions, the image measure $\gamma \circ u
^{-1}$ defined in ${\mathcal B}(\R)$ by $\gamma \circ u^{-1}(I) = \gamma (u^{-1}(I))$ has a smooth density $k$ with 
respect to the Lebesgue measure (e.g.,  \cite[Theorem 6.9.2]{Bog98Gau}). 
For each $r$ such that $k(r)>0$,   a  Radon measure  $\sigma_r$ supported on $u^{-1}(r)$ is well defined, and  we have
\begin{equation}\label{eqpart}
\int_{\{u< r\}} \divg v \,d\gamma = \int_{\{u =  r\}}\frac{[v,\nabla_{H}u]_H}{|\nabla_{H}u|_H}\, d\sigma_r,
\end{equation}
for all $v\in W^{\infty}(X, \gamma)$ such that $[v,\nabla_{H}u]_H$ is continuous. 

Since $k$ is continuous, for all $r\in\R$ we have
\begin{equation}\label{eqr}
k(r) = \lim_{\eps \to 0}\frac{1}{2\eps} \int_{r-\eps}^{r+\eps} k(s)\,ds =  
\lim_{\eps \to 0}\frac{\gamma\left(\{x: r-\eps \leq u(x) \leq r+\eps\}
\right)}{2\eps}.
\end{equation}
In particular, $\gamma(\{u
^{-1}(r)\})=0$ for all $r\in\R$. 
 
\begin{prop}\label{am}
Let $u
:X\to \R$ satisfy
\begin{equation}\label{eqF}
u
\in C(X)\cap \bigcap_{p>1,\, k\in \N} W^{k,p}(X,\gamma) ,
\qquad \nabla_H  
 u
\in C(X,H), 
\qquad \frac{1}{|\nabla_{H}
u
|_H} \in \bigcap_{p > 1} L^p(X,\gamma). 
\end{equation}
Then, for all $r\in\R$ with $k(r)>0$,  the level set $\{u
< r\}$ has finite 
perimeter and 
\begin{equation}\label{eqps}
P_\gamma(\{u
< r\}) = \sigma_r(u^	{-1}(r))   = \int_{\{u
<r\}}\divg \nu_H
\, d\gamma ,
\end{equation}
where $\nu_H
 := \nabla_H u/|\nabla_H u|_H$. 
\end{prop}
\begin{proof}
For all $\eps>0$ we define
\[v_{\eps}(x) :=  \frac{\nabla_{H}u
(x)}{\left(\eps + |\nabla_{H}u
(x)|_H^2\right)^{1/2}}
\ \in C(X,H)\cap \bigcap_{p > 1,\, k\in \N} W^{k,p}(X,H).
\]
By \eqref{eqpart} we have
\begin{equation}\label{eps}
\int_{\{u
< r\}} \divg v_\eps \,d\gamma = \int_{\{u
=  r\}} 
\frac{|\nabla_{H} u|_H}{(\eps + |\nabla_{H}
u|_H^2)^{1/2}}\, d\sigma_r.
\end{equation}
Letting $\eps \to 0$, by monotone convergence the right hand side goes to  
$$ \int_{\{u
=  r\}} \, d\sigma_r = \sigma_r(u^{-1}(r))\in [0,+\infty].$$
On the other hand, for a.e. $x\in X$ we have
\[
\divg v_\eps(x) =  \frac{\divg \nabla_H u(x)}{(\eps+|\nabla_{H}u|^2_H)^{1/2}} 
+ \frac{[\nabla^2_{H}u(x) \nabla_{H}u(x), \nabla_{H}u
(x)]_H}{(\eps + |\nabla_{H}u(x)|_H^2)^{3/2}} 
\]
so that  $\lim_{\eps\to 0} \divg v_\eps(x) = \divg \nu_H
(x)$ and, denoting by $\|\cdot \|_{HS}$ the Hilbert-Schmidt norm, 
\[
|\divg v_\eps(x)| \le 
\frac{|\divg \nabla_H u(x)|}{|\nabla_{H}u(x)|_H} 
+ \frac{\|\nabla^2_{H}u\|_{HS}}{|\nabla_{H}u(x)|_H}
\in \bigcap_{p > 1} L^p(X,\gamma).
\]
So, the left-hand side of \eqref{eps} goes to 
$$\int_{\{u<r\}}\divg \nu_H\, d\gamma ,$$
which implies that $\sigma_r(u
^{-1}(r))<+\infty$ and that  the second equality in \eqref{eqps} holds. 
Moreover,  by \eqref{eqpart}
\[
\begin{aligned}
P_\gamma(\{u
< r\}) &= \sup_{v\in {\mathcal F
C}^1_b(X,H),\ |v|_H\leq 1}
\int_{\{u
< r\}} \divg v \,d\gamma  
\\
&= \sup_{v\in {\mathcal F
C}^1_b(X,H),\ |v|_H\leq 1}\int_{\{u
=  r\}}
\frac{[v,\nabla_{H}
u
]_H}{|\nabla_{H}
u
|_H}\, d\sigma_r
\\
&\le \sigma_r(u
^{-1}(r))<+\infty.
\end{aligned}\]
In particular, the function $\chi_{\{u < r\}}$ belongs to $BV(X,\gamma)$. We claim that
\begin{equation}\label{olla}
D_\gamma\chi_{\{ u< r\}} = -\frac{\nabla_{H}u}{|\nabla_{H}u|_H}\,\sigma_r .
\end{equation}
In fact, \eqref{eqpappa} and \eqref{eqpart} imply
\begin{equation}\label{eqOnCylFunct}
\int_X [\varphi,\sigma_{ \{u<r\}}]_H d|D_\gamma \chi_{\{u<r\}}|=
-\int_X \frac{[\varphi,\nabla_H u]_H}{|\nabla_H u|}d\sigma_r,
\end{equation}
for any $\varphi\in {\cal F}C^{\infty}_b(X,H)$, the subset of $ {\cal F}C^{1}_b(X,H)$ consisting of smooth functions. We remark that ${\cal F}C^{\infty}_b(X,H)$ is dense in 
$C(K,H)$ for any compact set $K\subset X$. Indeed, for each $u\in C(K,H)$ the range of $u$
  is compact, so that for each $\varepsilon >0$ there is a finite dimensional subspace $Y$ of $H$ such that  $u(K) $ is contained in the $\varepsilon$--neighborhood of $Y$. The Stone--Weierstrass theorem yields the approximation of $\Pi u$, where $\Pi$ is the orthogonal projection on $Y$, and hence the approximation of $u$. 
Moreover, since $X$ is separable and the measures $|D_\gamma\chi_{\{u<r\}}|$
and $\sigma_r$ are finite, then they are tight, so that  
equality \eqref{eqOnCylFunct} can be extended to any $\varphi \in C_b(X,H)$, and the claim follows.  This yields 
\eqref{olla} and hence the first equality in \eqref{eqps}.
\end{proof}

We recall that  the set $\{r\in\R:\,k(r)>0\}$ is connected (\cite{HS}), and  it is dense in the range of $u$ since $u$ is continuous. Then, it contains the interior part of the range of $u$.

Notice that, for each $r$ in the range of $u$, at points $x$ such that $\nabla_H u\neq 0$ the vector   
$\nu_H = \nabla_H u/|\nabla_H u|_H$
is orthogonal to all tangent vectors to the level set $\{u=r\}$ belonging to $H$, with respect to 
the scalar product in $H$. 
Then, it may be considered as the (exterior) unit normal vector to the surface $\{u=r\}$.

Now  we extend a part of Proposition \ref{am} to a wider class of functions $u$. 
 
 \begin{prop}\label{disint1}
Assume that  $u\in BV(X,\gamma)\cap L^{p}(X,\gamma)$ and $z\in W^{1,p'}(X,\gamma;H)$ for some $p\in [1,+\infty)$ satisfy  $|z(x)|_H\leq 1$ for a.e. $x\in X$ and 
$$
| D_\gamma u |(X) = \int_{X} u\, \mathrm{div}_\gamma\, z\,d\gamma.
$$
Then 
\begin{equation}\label{disint2}
P_\gamma(\{u <r\}) = -  \int_{\{u <r\}} \mathrm{div}_\gamma\, z\,d\gamma
\le \Vert \mathrm{div}_\gamma\, z\Vert_{L^1(X,\gamma)}
\end{equation}
for all $r\in\R$. 
\end{prop}
\begin{proof}
Recalling that $\int_X \mathrm{div}_\gamma\, z\,d\gamma=0$,
by the Layer-Cake formula we get
\begin{eqnarray*}
| D_\gamma u |(X) = \int_{X} u \,\mathrm{div}_\gamma\, z\,d\gamma
= - \int_{-\infty}^\infty dr
\int_{\{u < r\}} \mathrm{div}_\gamma\, z\,d\gamma, 
\end{eqnarray*}
while by \eqref{eqcoarea}, 
\begin{eqnarray*}
| D_\gamma u |(X)  
= \int_{-\infty}^\infty P_\gamma(\{u <r\})\, d r .
\end{eqnarray*}
On the other hand, for all $r\in\R$ we have
\begin{equation}\label{ester}
- \int_{\{u < r\}} \mathrm{div}_\gamma\, z\,d\gamma   \leq P_\gamma( \{u<r\}).
\end{equation}
This follows approaching $z$ by  a sequence of vector fields $z_n \in {\mathcal FC}_b^1(X, H)$, with $|z_n|_H\leq 1$, in $W^{1,p'}(X, \gamma)$ if $p>1$, in any $W^{1,q}(X, \gamma)$ if $p=1$, and recalling \eqref{defDu}.  
 
Comparing the integrals that give $| D_\gamma u |(X) $,  for almost every $r\in\R$ we get
\begin{equation}
\label{r}
P_\gamma(\{u <r\}) = - \int_{\{u <r\}} \mathrm{div}_\gamma\, z\,d\gamma .
\end{equation}
Fix now any $r_0\in \R$, and let $(r_n)$ be a sequence of numbers such that  (\ref{r}) holds, $r_n <r_0$, and  $\lim_{n\to \infty}r_n =r$. Then $\lim_{n\to \infty} \chi_{\{u <r_n\}}(x)=  \chi_{\{u <r_0\}}(x)$ for each $x\in X$, so that by dominated convergence $ \lim_{n\to \infty}  \int_{\{u <r_n\}} \mathrm{div}_\gamma\, z\,d\gamma = \int_{\{u <r_0\}} \mathrm{div}_\gamma\, z\,d\gamma$. 
By the lower semicontinuity of $P_\gamma$, we obtain
\begin{equation}\label{aster}
P_\gamma(\{u <r_0\}) \leq 
\liminf_{n\to \infty}  P_\gamma(\{u <r_n\}) = -\liminf_{n\to \infty}  \int_{\{u <r_n\}} \mathrm{div}_\gamma\, z\,d\gamma = -\int_{\{u <r_0\}} \mathrm{div}_\gamma\, z\,d\gamma.
\end{equation}
\eqref{disint2} now follows from \eqref{ester} and \eqref{aster}.
\end{proof}

\begin{cor}
\label{cor:perimetri}
Assume that
$$
u\in W^{1,1}(X,\gamma;H)\cap L^{p}(X,\gamma), \qquad 
\nabla_H
 u\ne 0\ \gamma-{\rm a.e.}, \qquad 
\frac{\nabla_H
 u}{|\nabla_H
 u|_H}\in W^{1,p'}(X,H), 
$$ 
for some $p\in [1,+\infty)$. Then for each $r\in \R$ we have
 \begin{equation}
\label{formula}
P_\gamma(\{u <r\}) =  \int_{\{  u<r \}}  {\rm div}_\gamma \nu_H \, d\gamma ,   
\end{equation}
the function 
$r\mapsto P_\gamma(\{u <r\})$ is continuous in $\R$, and 
$$\lim_{r\downarrow \essinf   u} P_\gamma(\{u <r\}) =0, \quad \lim_{r\uparrow \esssup   u} P_\gamma(\{u <r\}) =0.$$
In particular, $P_\gamma(\{u <r\})$ is bounded by a constant independent of $r$.
\end{cor}
\begin{proof}
The function $u$ satisfies the assumptions of Proposition \ref{disint1} with  $z= -\nabla_H
 u/|\nabla_H u|_H$, hence formula \eqref{formula}
 holds. 
For any $r_0\in \R$ we have
$$\lim_{r\uparrow r_0} \chi_{\{u<r\}}(x) =  \chi_{\{u<r_0\}}(x), \quad  \lim_{r\downarrow r_0} \chi_{\{u<r\}}(x) =  \chi_{\{u\leq r_0\}}(x), \quad \forall x\in X.$$
Since $u\in W^{1,1}(X, \gamma)$ and $\nabla_Hu\neq 0$ a.e., then $\gamma(\{u=r_0\})=0$   by \cite[Thm. 9.2.4]{Boga}.  So, $\lim_{r\rightarrow r_0} \chi_{\{u<r\}} =  \chi_{\{u<r_0\}}$ a.e., and 
 \eqref{formula} yields that $r\mapsto P_\gamma(\{u <r\})$ is continuous at $r_0$ by dominated convergence. 
In particular, since $ \chi_{\{u<\essinf u\}} =0$ a.e., then $\lim_{r\downarrow \essinf   u} P_\gamma(\{u <r\}) =0$. 
Since $ \chi_{\{u<\esssup u\}} = 1$ a.e., then $ \lim_{r\uparrow \esssup   u} = \int_X  \mathrm{div}_\gamma\,\nu_Hd\gamma =0$, and the statement follows. 

\end{proof}
\subsection{Balls in   Hilbert spaces}\label{Sect:Balls}

Assume now that $X$ is a Hilbert space with the inner product $\scal{\cdot}{\cdot}$ and the induced norm $\|\cdot\|$.  
The Hilbert space $H$ is just $Q^\frac 1 2X$ equipped with the inner product  
$[h,k]_H=\langle Q^{-\frac 1 2}h, Q^{-\frac 1 2}k\rangle$. We choose  a orthonormal basis of 
$X$ consisting of eigenvectors $e_j$ of $Q$, $Qe_j=\lambda_j e_j$ for $j\in \N$. Then setting 
$h_j=e_j/|e_j|_H=\sqrt\lambda_j e_j$, the elements $h_j$, $j\in \N$,  form a orthonormal basis 
of $H$. 
%
%
%
%
%
Fixed any $x_0\in X$ and $r>0$, the the exterior unit normal vector $\nu_H
$ at $x\in \partial B_r(x_0)$ is given by
$$\nu_H
(x) =  \frac{Q (x-x_0)}{|Q (x-x_0)|_H}
=\frac{Q (x-x_0)}{\|Q^{1/2}(x-x_0)\|}.$$
We define the mean curvature at $x\in \partial B_r(x_0)$ as the divergence of $\nu_H
$ at $x$. Since 
$[\nu_H(x),h_j]_H = \scal{x-x_0}{h_j}/\|Q^{1/2}(x-x_0)\|$, we obtain 
\begin{align} 
 \label{kga}\divg \nu_H
 (x)= & 
\sum_{j=1}^\infty \left[ \partial_{h_j} \left( 
\frac{\scal{ x-x_0}{h_j}}{(\sum_{i=1}^{\infty}\scal{x-x_0}{h_i} ^2)^{1/2}}
\right)- 
\frac{\scal{ x }{h_j}}{\lambda_j}\cdot  \frac{ \scal{ x-x_0 }{ h_j}}{\|Q^{1/2}(x-x_0)\|}\right]\\
\nonumber
=&
\sum_{j=1}^\infty \frac{\|h_j\|^2}{\|Q^{1/2}(x-x_0)\|} 
-\sum_{j=1}^\infty
\frac{\scal{ x-x_0 }{h_j}^2\|h_j\|^2}{\|Q^{1/2}(x-x_0)\|^{3}} -
\sum_{j=1}^\infty
\frac{\scal{x-x_0}{e_j}\scal{x}{e_j}}{\|Q^{1/2}(x-x_0)\|}
\\
\nonumber
=& \frac{{\rm Tr}\,Q-r^2-\langle x-x_0,x_0\rangle}{\langle Q(x-x_0),(x-x_0)\rangle^\frac 1 2} -
\frac{\| Q(x-x_0)\|^2}{\langle Q(x-x_0),(x-x_0)\rangle^\frac 3 2}
\end{align}
where 
${\rm Tr}\,Q=\sum_{i\geq 1}\lambda_i<+\infty$ is the trace of the covariance operator $Q$.
By computing $\divg\nu_H
$ at $x=re_j$, we see that the mean curvature 
is  unbounded if $X$ is infinite--dimensional.  

\begin{lemma}\label{fieldballs}
If $X$ is an infinite dimensional Hilbert space, 
the function $u(x)=\|x-x_0\|^2$ satisfies condition \eqref{eqF} 
for all $x_0\in X$.
\end{lemma}

\begin{proof}
Being $X$ a Hilbert space, $u\in C^\infty(X,\R)$. Since $\nabla_Hu(x) = 2Q(x-x_0)$,  $u\in W^{k,p}(X, \gamma)$ for each $p>1$, $k\in \N$. Therefore it is enough to show that
\begin{equation}\label{curvature1}
\left\|\frac{1}{|\nabla_H
u|_H}\right\|_{L^p(X,\gamma)} = 
\frac{1}{2}\left\|\langle  Q(\cdot -x_0),\cdot -x_0\rangle^{-\frac 1 2}\right\|_{L^p(X,\gamma)} < +\infty
\qquad \forall p> 1.
\end{equation}
For all $x\in X$ and $n\in\mathbb N$ we have
\[
\left|\frac{1}{\langle Q(x-x_0),(x-x_0)\rangle}\right|^\frac p 2 \leq
\left|\frac{1}{\sum_{k=1}^n \lambda_k\langle x-x_0,e_k\rangle^2}\right|^\frac p 2
\leq \left|\frac{1}{\lambda_n\,\sum_{k=1}^n \langle x-x_0,e_k\rangle^2}\right|^\frac p 2.
\]
Therefore
\[
\left\|\langle Q(\cdot -x_0),\cdot -x_0\rangle^{-\frac 1 2}\right\|_{L^p(X,\gamma)}
\le \lambda_n^{-\frac{p}{2}}\,\int_{H_n}\left|\sum_{k=1}^n
\langle x-x_0,e_k\rangle^2\right|^{-\frac{p}{2}}\,d\gamma_n(x)
<+\infty
\]
for all $n>p$, which gives \eqref{curvature1}.
\end{proof}

Using Corollary \ref{cor:perimetri} we get several properties of perimeters of balls. 

\begin{cor}
\label{corfinal}
Every ball in $X$ has finite perimeter; for each $x_0\in X$ the function $r\mapsto p(r):=P_\gamma(B_r(x_0))$ is continuous in $[0,+\infty)$ and 
$$
\lim_{r\to 0} P_\gamma(B_r(x_0))= \lim_{r\to +\infty} P_\gamma(B_r(x_0))= 0.
$$
Moreover,
there exist $0<r_m<r_M$ (depending on $x_0$) such that
$p$ is monotone increasing in $[0,r_m]$ and monotone decreasing
in $[r_M,+\infty)$.
\end{cor}
\begin{proof}
The function   $u(x)=\|x-x_0\|^2$ satisfies the assumptions of Corollary \ref{cor:perimetri}. 
Indeed, recalling \eqref{kga}, for $ \|x-x_0\|= r$  we have the estimate
\begin{equation}\label{kgabis}
\frac{{\rm Tr}\,Q-r^2 - \| x_0\|r-\lambda_1}{\langle Q(x-x_0),(x-x_0)\rangle^\frac 1 2}
\le \divg \nu_H
 \le
\frac{{\rm Tr}\,Q-r^2 + \| x_0\|r}{\langle Q(x-x_0),(x-x_0)\rangle^\frac 1 2}\,, 
\end{equation}
that gives an $L^p$-bound on $\divg \nu_H
$, namely $\divg \nu_H \in L^p(X, \gamma)$ for  $p<n$ if $X$ is $n$-dimensional, $n\geq 2$, and 
$\divg \nu_H \in \cap_{p>1}L^p(X,\gamma)
$ if $X$ is infinite dimensional, by Lemma \ref{fieldballs}. 

Except for the last claim, the statement is a consequence of Corollary \ref{cor:perimetri}. 
Moreover,  \eqref{kgabis} implies that $ \divg \nu_H
 \geq 0$ if $\|x-x_0\| \leq r_0 := (-\|x_0\|+\sqrt{\|x_0\|^2+4({\rm Tr}\,Q-\lambda_1)})/2$ and $ \divg \nu_H
 \leq 0$ if $\|x-x_0\| \geq r_1:= ( \|x_0\|+\sqrt{\|x_0\|^2+4{\rm Tr}\,Q})/2$. By \eqref{formula}, 
$p$ is monotone increasing in $(0, r_0]$ and monotone decreasing in $[r_1, +\infty)$. 
The last claim follows. 
\end{proof}

We point out that a continuity result similar to the one of Corollary \ref{corfinal}
was already proved by Talagrand in \cite{talagrand},
where the notion of perimeter is replaced by the density of the distribution
of the norm of $X$.

\begin{rmk1}
Similar results are easily available for ellipsoids $\{ x\in X: \|T(x-x_0)\| = r\}$, if $T\in L(X)$ is the diagonal operator 
$$Tx = \sum_{k=1}^{\infty} t_k\langle x, e_k\rangle e_k$$
 with $t_k\geq 0$ for each $k\in \N $ and $t_k>0$ for infinite values of $k$. Indeed, setting $u(x) = \|T(x-x_0)\| ^2$ we have $\nabla_H u(x) = 2QT(x-x_0)$, $\nu_H (x) = QT(x-x_0)/ \|Q^{1/2}T(x-x_0)\|$, 
 Lemma \ref{fieldballs} goes through with obvious changements, and 
 $$  \divg \nu_H =  \frac{{\rm Tr}\,QT - \langle x, T(x-x_0)\rangle}{\|Q^{1/2}T(x-x_0)\|} -
\frac{\| QT^{3/2}(x-x_0)\|^2}{\|Q^{1/2}T(x-x_0)\|^3}$$
so that  
$$| \divg \nu_H| 
 \le\frac{{\rm Tr}\,QT+\|x\|\,  \|T(x-x_0)\|+  \|Q T \|_{L(H)}  }{\| Q^{1/2}T(x-x_0)\|}$$
 which implies that $ \divg \nu_H\in L^p(X, \gamma)$ for every $p<\infty$. So, Corollary \ref{cor:perimetri} may be applied and it yields that the function $r\mapsto p(r):=P_\gamma(\{ x\in X: \|T(x-x_0\| = r\} )$ is continuous in $[0,+\infty)$ and it vanishes as $r\to 0$ and as $r\to +\infty$. 
\end{rmk1}

\section{Open convex sets have finite perimeter}\label{secopen}

In this section we prove that any open convex subset of an abstract Wiener space $X$ has 
finite perimeter.

We first recall an important property of Gaussian measures in separable spaces \cite[Thm. 4.2.2, Rem. 4.2.5]{Bog98Gau}.

\begin{prop}\label{proisop}
Let $A,\,B$ be Borel subsets of $X$ and let $\lambda\in [0,1]$. Then
\[
\gamma (\lambda A+(1-\lambda)B) \ge \gamma(A)^\lambda\gamma(B)^{1-\lambda} .
\]
\end{prop}

\begin{prop}\label{propen}
For every  open convex subset  $C\subset X$,  $\gamma(\partial C)=0$
and $P_\gamma(C)<+\infty$.
\end{prop}

\begin{proof}
The statement is obvious if $C=X$, so we assume $C\neq X$. Moreover,  without loss of generality we may assume that  $C$ contains the origin of $X$.
Indeed, if $0\notin C$ there exists $h\in C\cap H$ such that $[h,h']_H\ge 0$ for all $h'\in C\cap H$. 
If  $h=Rx^*$,  then $\langle x, x^*\rangle \geq 0$ for each $x\in C$, since $H$ is dense in $X$, and using 
the Cameron--Martin formula \eqref{absCont}   we get
\[
\begin{aligned}
\gamma(C)&\le e^{ \frac{|h|^2_H}{2}}\gamma(C-h),
\\
P_\gamma(C)&\le e^{ \frac{|h|^2_H}{2}}P_\gamma(C-h).
\end{aligned}
\]
Hence, it is enough to prove that the statement holds when $C$ is replaced by $C-h$. 
In this case, since $C$ is open,  
it contains a ball of radius $r>0$ centered at the origin.

We argue as in \cite[Proposition 3.2]{linde} (see also \cite{HSD})
and we set $g(t):=\gamma(tC)$, for $t\in [0,+\infty)$.
For  $0\le t_1\le t_2$, applying Proposition \ref{proisop}
  with $A=t_1 C$ and $B=t_2 C$, we get
\begin{equation}\label{isopeq}
g(\lambda t_1+(1-\lambda)t_2)\ge g(t_1)^\lambda g(t_2)^{1-\lambda}
\end{equation}
for all $\lambda\in [0,1]$.
Inequality \eqref{isopeq}
implies  that the function $f(t):=\log g(t)$ is concave on $(0,+\infty)$.
An immediate consequence is that $g$ is continuous on $(0,+\infty)$, 
so that $\gamma(\partial C)=0$,
since $\partial C \subset (1+\eps)C\setminus C$ for every $\eps >0$.
Moreover, there exists $M>0$ such that
\[
\frac{g(1+\eta)-g(1)}{\eta}\le M,  \qquad \forall \eta \in (0,1).
\]
Letting
\[
d_C(x):= \inf_{y\in C} \|x-y\|, 
\qquad x\in X,
\]
the distance  $d_C$ is $1$--Lipschitz in $X$, and then it is $H$--Lipschitz with constant $c_H$
as in \eqref{contEmb}. By \cite[Example 5.4.10]{Bog98Gau}, $d_C\in W^{1,1}(X, \gamma)$. Moreover, $|\nabla_H
 d_C|_H\leq c_H$ a.e., and consequently for every $\eta \in (0, 1)$
\[
\begin{aligned}
M\ge \frac{g(1+\eta)-g(1)}{\eta} =& \,\frac{1}{\eta}\,\left(\gamma((1+\eta)C)-\gamma(C)\right)
\\
\ge& \,\frac{1}{\eta}\,\left(\gamma(C+ B_{\eta r}) - \gamma(C)\right)
\\
\ge& \,\frac{1}{\eta c_H}\int_{(C+ B_{\eta r})\setminus C}|\nabla_H
 d_C|_H\,d\gamma .\end{aligned}
\]
For every $\eta \in (0, 1)$ and $r>0$ the function 
\[
u(x):= \min\{d_C(x),\eta r\},
\]
is still $H$--Lipschitz and consequently in $ W^{1,1}(X, \gamma)$, moreover
$\nabla_Hu = \chi_{(C+ B_{\eta r})\setminus C}\nabla_H
 d_C$ and for $t\leq \eta r = \max u $ we have $\{ u<t\} = C+B_t$. Applying  \eqref{eqcoarea} we obtain
\[\begin{aligned}
 \frac{1}{\eta c_H}\int_{(C+ B_{\eta r})\setminus C}|\nabla_H
 d_C|_H\,d\gamma = & \, 
 \frac{1}{\eta c_H} \int_X |\nabla_H u|_H\,d\gamma
\\  
=   & \,
 \frac{r}{c_H}\,\frac{1}{\eta r}  \int_0^{\eta r}P_\gamma(C+ B_t)\,dt
\end{aligned}
\]
As a consequence, letting $\eta\to 0 $, there exists a decreasing sequence $t_n\to 0 $ such that
\[
P_\gamma(C+B_{t_n}) \le \frac{M c_H}{r}.
\]
The statement follows from the lower semicontinuity of the perimeter.
\end{proof}

\section{A compact convex set with infinite perimeter}\label{secinfty}

In this section $X$ is an infinite dimensional Hilbert space, and we prove that there exists   a convex set
  with positive measure and infinite perimeter.
 
\begin{prop}\label{PropConvC}
There exists a closed convex set $C\subset X$ with  $P_\gamma(C)=+\infty$.
Moreover, if the covariance operator $Q$ satisfies
\begin{equation}\label{covCond}
\sum_{j\in \N}\lambda_j \log j<+\infty,
\end{equation}
then  $C$ is compact.
\end{prop}

\begin{proof}
We fix a sequence of positive numbers $r_k>0$ satisfying 
\begin{equation}\label{asympRj}
\sqrt{\frac{2}{\pi}} \frac{e^{-\frac{r^2_k}{2}}}{r_k}=
\frac{1}{(k+1)(\log (k+1))^{\frac{3}{2}}}.
\end{equation}
Then for $k$ big enough we have
\begin{equation}\label{rjGrowth}
(\log (k+1))^\frac{1}{2} \leq r_k\leq 2(\log (k+1))^\frac{1}{2}
\end{equation}
We define the sets
$$
Q_n :=\{x\in H_n: -r_j\leq [x,h_j]_H\leq r_j, \forall j\leq n\}
\qquad
C_n:=\Pi_n^{-1}(Q_n)\subset X.
$$
The decomposition $X={\rm Ker}(\Pi_n)\oplus H_n$ and
$\gamma=\gamma_n\otimes \gamma_n^\perp$ yields
\begin{align}\label{measureCn}
\gamma(C_n)=&\prod_{k=1}^n \sqrt{\frac{2}{\pi}} \int_0^{r_k}e^{-\frac{s^2}{2}}ds
\geq \prod_{k=1}^n \left(
1-\frac{1}{(k+1)(\log (k+1))^{\frac{3}{2}}}
\right)\\
\nonumber
=&
\exp\left(
\sum_{k=1}^n \log \Big(
1-\frac{1}{(k+1)(\log (k+1))^{\frac{3}{2}}},
\Big)\right)
\end{align}
where we have used the inequality
\[
\int_{\varrho}^{+\infty} e^{-\frac{s^2}{2}}ds \leq \frac{1}{\varrho}\int_{\varrho}^{+\infty}s e^{-\frac{s^2}{2}}ds =
\frac{e^{-\frac{\varrho^2}{2}}}{\varrho}.
\]
The sequence $C_n$ converges  decreasing to the closed set
$$
C:=\bigcap_{n\in \N}C_n
$$
and
\[
\gamma(C)=
\lim_{n\to +\infty} \gamma(C_n)
\geq
\exp\left(
\sum_{k=1}^\infty \log \Big(
1-\frac{1}{(k+1)(\log (k+1))^{\frac{3}{2}}}
\Big)\right):=a .
\]
The series in the exponential is  asymptotically equivalent to the series
\[
\sum_{k=1}^\infty \frac{1}{(k+1)(\log (k+1))^{\frac{3}{2}}}.
\]
which is convergent. Then, $a>0$. 

We now estimate the perimeters of the sets $C_n$; we denote by $Q_n^k\subset H_{n}$
and $C_n^k\subset X$ the sets
$$
Q_n^k:=\{x\in H_n: -r_j\leq [x,h_j]_H\leq r_j, \forall j\leq n, j\neq k\}
\qquad
C_n^k:=\Pi_{n}^{-1}(Q_n^k).
$$
Then, since $\gamma(C_n^k)\geq \gamma(C_n)\geq a$, by \eqref{asympRj}
$$
P_{\gamma}(C_n)=
\sum_{k=1}^n \sqrt{\frac{2}{\pi}}e^{-\frac{r_k^2}{2}}\gamma(C_n^k)
\geq a\sqrt{\frac{2}{\pi}}\sum_{k=1}^n e^{-\frac{r_k^2}{2}}
=a\sum_{k=1}^n \frac{r_k}{(k+1)(\log (k+1))^{\frac{3}{2}}}
\to +\infty.
$$
It remains to prove that the perimeter of $C$ is the limit of the perimeters of $C_n$.
To  this aim we   consider the conditional expectations
$$
u_{m,n}(x)={\mathbb E}_m(\chi_{C_n})(x) = \int_X \chi_{C_n}(\Pi_m x+ (I-\Pi_m)y)\gamma(dy)
,\qquad n> m.
$$
A direct computation shows that $u_{m,n}=\alpha_{m,n}\chi_{C_m}$ with
$$
\alpha_{m,n}=\prod_{j=m+1}^n \sqrt{\frac{2}{\pi}}\int_0^{r_j} e^{-\frac{s^2}{2}}ds.
$$
As $n\to +\infty$, since $\chi_{C_n}\to \chi_C$ in $L^1(X,\gamma)$
and by the continuity of the conditional expectation ${\mathbb E}_m$
in $L^1(X,\gamma)$, we obtain  $u_{m,n}\to u_m={\mathbb E}_m(\chi_C)$
in $L^1(X,\gamma)$ with
$$
u_m=\alpha_m\chi_{C_m},\qquad
\alpha_m=\prod_{j=m+1}^{+\infty} \sqrt{\frac{2}{\pi}}\int_0^{r_j} e^{-\frac{s^2}{2}}ds
$$
Let us show that $\lim_{m\to \infty} \alpha_m= 1$. We have  
$$
\alpha_{m,n}\geq \prod_{j=m+1}^n\left(
1-\frac{1}{(j+1)(\log (j+1))^{\frac{3}{2}}}
\right)
\sim
\exp\left(
-\sum_{j=m+1}^n\frac{1}{(j+1)(\log (j+1))^{\frac{3}{2}}}
\right)
$$
so that the assertion follows from the convergence of the series in the right--hand side.
Then, 
$$
P_\gamma(C)=\lim_{m\to +\infty} |D_\gamma u_m|(X)
=\lim_{m\to +\infty}P_\gamma(C_m)=+\infty.
$$
Moreover, if
\begin{equation}\label{condRjLj}
\sum_{j\in \N} r_j^2 \left\|h_j\right\|^2 <+\infty\,, 
\end{equation}
then $C$ is compact. 
Recalling \eqref{rjGrowth}, since $\|h_j\|^2=\lambda_j$, \eqref{condRjLj} is equivalent to \eqref{covCond}.
\end{proof}

\noindent Notice that,   taking $r\le\min_k r_k$,   the ball centered at $0$ with radius $r$
of the Cameron--Martin space $H$ is contained in $C$.

\medskip\noindent {\bf Acknowledgements.}
V. Caselles and M. Novaga
acknowledge partial support by Acci\'on Integrada Hispano Italiana HI2008-0074.
V. Caselles also acknowledges support  by MICINN project,
reference MTM2009-08171, by GRC reference 2009 SGR 773
and by  ''ICREA A\-ca\-d\`e\-mia''  for excellence in research, the last two funded by the
Generalitat de Catalunya. A. Lunardi was supported by PRIN 2008 research project 
``Deterministic and stochastic methods in evolution problems".  
M. Miranda and M. Novaga acknowledge partial support by the GNAMPA project
``Metodi geometrici per analisi in spazi non Euclidei;
spazi metrici doubling, gruppi di Carnot e spazi di Wiener'', and by the Research Institute Le Studium.



\end{document}